\documentclass[12pt]{amsart}
\numberwithin{equation}{section}
\newtheorem{theorem}{Theorem}[section]
\newtheorem{lemma}[theorem]{Lemma}
\newtheorem*{lemma*}{Lemma}
\newtheorem{corollary}[theorem]{Corollary}

\newtheorem{conjecture}[theorem]{Conjecture}
\newtheorem{question}[theorem]{Question}

\theoremstyle{definition}
\newtheorem{definition}[theorem]{Definition}

\theoremstyle{remark}

\newcommand{\N}{\mathbb{N}}

\newcommand{\Z}{\mathbb{Z}}

\newcommand{\Q}{\mathbb{Q}}
\newcommand{\F}{\mathbb{F}}

\newcommand{\Fx}{\mathbb{F}_2[x]}
\newcommand{\Fpx}{\mathbb{F}_p[x]}

\newcommand{\Zx}{\mathbb{Z}[x]}
\newcommand{\Res}{\text{Res{\hskip 1pt}}}

\usepackage{hyperref}
\usepackage{amssymb}
\usepackage{caption}
\usepackage{fullpage}
\usepackage{amsmath}
\usepackage{amsfonts}
\usepackage{latexsym}

\title{The Distance to a Squarefree Polynomial Over $\mathbf \Fx$}

\author[1]{Michael Filaseta }
\address{University of South Carolina\\ Department of Mathematics\\ Columbia, SC 29208}
\email{filaseta@math.sc.edu}
\author[2]{Richard A. Moy}
\address{Lee University\\ Department of Mathematical Sciences\\ Cleveland, TN 37320}
\email{rmoy@leeuniversity.edu}
\begin{document}

\begin{abstract}
In this paper, we examine how far a polynomial in  $\Fx$ can be from a squarefree polynomial. 
For any $\epsilon>0$, we prove that for any polynomial $f(x)\in\Fx$ with degree $n$, 
there exists a squarefree polynomial $g(x)\in\Fx$ such that $\deg g \le n$ and $L_{2}(f-g)<(\ln n)^{2\ln(2)+\epsilon}$
(where $L_{2}$ is a norm to be defined).  As a consequence, the
analagous result holds for polynomials $f(x)$ and $g(x)$ in $\Z[x]$. 
\end{abstract}
\maketitle

\section{Introduction}
In the 1960's, P\'al Tur\'an  (cf.~\cite{S67}) posed the problem of determining whether there is an absolute constant $C$ such that 
for every polynomial $f(x)=\sum_{j=0}^n{a_{j} x^{j}} \in \mathbb Z[x]$, 
there is a polynomial $g(x)=\sum_{j=0}^n{b_{j}x^{j}} \in \mathbb Z[x]$ irreducible over the rationals satisfying 
$L(f-g):= \sum_{j=0}^{n}{|b_j-a_j|} \le C$.  
It is currently known that the existence of such a $C$ is connected to an open problem on covering systems of the integers with distinct odd moduli \cite{F2002,S67}; 
if one allows $g(x)$ to have degree $> n$, then one can take $C = 3$ \cite{BF10, S70};  
for all $f(x)$ of degree $\le 40$ such a $g(x)$ exists with $C = 5$ \cite{FM12};  
for the corresponding problem in $\Fx$, if $C$ exists, then $C \ge 4$ \cite{BF10}; and
for the corresponding problem in $\Fpx$ with $p$ an odd prime, if $C$ exists, then $C \ge 3$ \cite{F14}.
Other papers on this topic include \cite{BH97,FM12,G94,LRW07,M10}. 
In \cite{F14}, a case is made for the following conjecture.

\begin{conjecture}\label{conj1pt1}
For every $f(x)\in\Z[x]$ of degree $n \ge 1$, there is an irreducible polynomial $g(x)\in\Z[x]$ of degree at most $n$ satisfying $L(f-g)\le 2$.
\end{conjecture}

In \cite{DS18}, Dubickas and Sha investigated an interesting variant of this conjecture where they asked how far a polynomial $f(x) \in\Z[x]$ can be from a squarefree polynomial, 
that is from a polynomial in $\Z[x]$ not divisible by the square of an irreducible polynomial over $\Q$.

\begin{conjecture}\label{conj1pt2}
For every $f(x)\in\Z[x]$ of degree $n \ge 0$, there is a squarefree polynomial $g(x)\in\Z[x]$ of degree at most $n$ satisfying $L(f-g)\le 2$.
\end{conjecture}

Among other nice results, Dubickas and Sha \cite[Theorem~1.4]{DS18}  show that if $g(x)$ is allowed to have degree $> n$, then such a squarefree polynomial $g(x)\in\Z[x]$ 
exists satisfying $L(f-g)\le 2$.
They \cite[Theorem~1.3]{DS18} also show that for $n \ge 15$, there are infinitely many polynomials $f(x) \in \mathbb Z[x]$ of degree $n$ such that 
if $g(x) \in \mathbb Z[x]$ is squarefree, then $L(f-g) \ge 2$.  We show in the next section that this latter result extends to $k$-free polynomials.

\begin{theorem}\label{kfree}
Let $k$ be an integer $\ge 2$.  
There exists a computable $N_{0} = N_{0}(k)$ such that if $n \ge N_{0}$, then there are 
infinitely many polynomials $f(x) \in \mathbb Z[x]$ of degree $n$ such that 
if $g(x) \in \mathbb Z[x]$ is $k$-free, then $L(f-g) \ge 2$.
\end{theorem}

Our argument for Theorem~\ref{kfree} gives as a permissible value of $N_{0}$ the number
\[
N_{0} = k \sum_{j=1}^{2k} (p_{j}-1) + k + 1,
\]
where $p_{1}, \ldots, p_{2k}$ are the first $2k$ primes.  We expect much smaller $N_{0}$ will suffice.  

One can approach the above conjectures by investigating the analogous questions for polynomials over finite fields.  Indeed, this is done for Conjecture~\ref{conj1pt1} in
\cite{BH97,F14,FM12,LRW07,M10}.  

\begin{definition}
Let $\F_p$ be any finite field with $p$ elements where $p$ is a prime. 
For any polynomial $f(x)\in\Fpx$, define its {\it length} $L_p(f)$ by choosing each of its coefficients in the interval $(-p\slash 2, p\slash 2]$ 
and then summing their absolute values in $\Z$.
\end{definition}

Using this definition of distance in $\Fpx$, Dubickas and Sha \cite[Question 6.2]{DS18} asked the following question.

\begin{question}
For any prime number $p$ and any polynomial $f(x)\in\F_p[x]$, is there a squarefree polynomial $g(x)\in\F_p[x]$ of degree at most $\deg f$ satisfying $L_p(f-g)\le 2$?
\end{question}

In this paper, we will prove the following theorem.

\begin{theorem}\label{thm:main1} 
Fix $\epsilon>0$. 
Let $f(x) \in \Fx$ with $\deg f = n$. 
If $n$ is sufficiently large, then there exists a squarefree polynomial $g(x) \in \Fx$ of degree $n$ such that 
\[
L_2(f-g) \le (\ln n)^{2\ln(2)+\epsilon}.
\]
\end{theorem}

In the next section, we justify the following consequence of Theorem \ref{thm:main1}.

\begin{corollary}\label{maincorollary} 
Fix $\epsilon>0$. Let $f(x)\in\Zx$ with $\deg f = n$. 
If $n$ is sufficiently large, then there exists a squarefree polynomial $g(x)\in\Zx$ of degree $n$ such that 
\[
L(f-g) \le (\ln n)^{2\ln(2)+\epsilon}.
\]
\end{corollary}

\section{Proofs of Theorem~\ref{kfree} and Corollary~\ref{maincorollary}}

Before turning to our main result, we establish Theorem~\ref{kfree} and show that Corollary~\ref{maincorollary}
is a consequence of Theorem \ref{thm:main1}.

\begin{proof}[Proof of Theorem~\ref{kfree}]
Fix a positive integer $k$.  
Let $\Phi_{n}(x)$ denote the $n$th cyclotomic polynomial.  
For distinct positive integers $m$ and $n$, 
Diederichsen \cite{FED1940} obtained the value of the resultant $\Res (\Phi_{n}(x), \Phi_{m}(x))$.  
For our purposes, we only use that this resultant is $1$ in the case that $m$ and $n$ are distinct primes.
For monic polynomials $f(x)$ and $g(x)$, one can view the $|\Res (f(x), g(x))|$ as the product
of $g(\alpha)$ as $\alpha$ runs through the roots of $f(x)$.  
It follows that for distinct primes $p$ and $q$, we have 
\[
\Res(\Phi_{p}(x)^{k}, \Phi_{q}(x)^{k}) =  \pm 1.
\]
Furthermore, for any prime $p$, one can see that
\[
\Res(x^{k}, \Phi_{p}(x)^{k}) = \pm 1.
\]
Both of the above resultants hold with $\pm 1$ replaced by $1$, but this is not important to us.

Let $p_{1}, p_{2}, \ldots, p_{2k}$ be arbitrary distinct primes.  
Define 
\[
f_{0}(x) = x^{k} 
\quad \text{ and } \quad
f_{j}(x) = \Phi_{p_{j}}(x)^{k} \ \ \text{for } 1 \le j \le 2k.
\]
From the above, we have $\Res(f_{i}(x), f_{j}(x)) =  \pm 1$ for distinct $i$ and $j$ in $\{ 0, 1, \ldots, 2k \}$.  
The significance of this is that as a consequence each $f_{i}(x)$ has an inverse modulo $f_{j}(x)$ in $\mathbb Z[x]$.  
Thus, a Chinese Remainder Theorem argument implies that for arbitrary $a_{j}(x) \in \mathbb Z[x]$, there is a 
$g(x) \in \mathbb Z[x]$ that satisfies
\[
g(x) \equiv a_{j}(x) \pmod{f_{j}(x)}, \quad \text{for all } j \in \{ 0, 1, \ldots, 2k \}.
\]

We set
\[
a_{0} = 0
\quad \text{ and } \quad
a_{j}(x) = (-1)^{j} x^{\lfloor (j-1)/2 \rfloor} \ \text{ for } 1 \le j \le 2k.
\]
Then $g(x)$ above has the property that $g(x) - (-1)^{j} x^{\lfloor (j-1)/2 \rfloor}$ is divisible by $f_{j}(x) = \Phi_{p_{j}}(x)^{k}$ 
for $1 \le j \le 2k$.  Furthermore, for any $\ell \ge k$, the condition $a_{0} = 0$ implies $g(x)$ and 
$g(x) \pm x^{\ell}$ are divisible by $x^{k}$.  
Taking $N$ equal to the degree of
\[
P(x) = \prod_{1 \le j \le 2k} \Phi_{p_{j}}(x)^k,
\]
we can find $g(x)$ as above of degree $< N + k$.  
Then for $n \ge N_{0} := N+k+1$ and arbitrary integers $a$ and $b$, the polynomial
\[
F(x) = g(x) + x^{n-N-1} P(x) (ax + b)
\]
of degree $n$ has the property that if $h(x) \in \mathbb Z[x]$ and $L(F-h) \le 1$, then $h(x)$ is divisible by one of
the $f_{j}(x)$ and, hence, not $k$-free.  The role of the expression $ax + b$ in the definition of $F(x)$ 
is to clarify that for a given $n \ge N_{0}$, there are infinitely many possibilities for $F(x)$, 
completing the proof of Theorem~\ref{kfree}.
\end{proof}

\begin{proof}[Proof of Corollary~\ref{maincorollary} assuming Theorem~\ref{thm:main1}]
We consider $\epsilon > 0$ and $n$ sufficiently large.  
Let $f_{2}(x) = f(x)$ if the leading coefficient of $f(x)$ is odd; otherwise, let 
$f_{2}(x) = f(x) + x^{n}$.  Thus, in either case, $f_{2}(x)$ has degree $n$ and an odd leading coefficient. 
Let $\bar{f}_{2}(x)$ be a $0,1$-polynomial (a polynomial all of whose coefficients are $0$ or $1$) satisfying 
$\bar{f}_{2}(x) \equiv f_{2}(x) \pmod{2}$.   
By Theorem~\ref{thm:main1}, there is a $0,1$-polynomial $\bar{g}_{2}(x)$, squarefree in $\Fx$, 
such that 
\[
L(\bar{f}_{2}-\bar{g}_{2}) = L_2(\bar{f}_{2}-\bar{g}_{2})<(\ln n)^{2\ln(2)+\epsilon/2}.
\]
Furthermore, $\bar{g}_{2}(x)$ has degree $n$ and, hence, an odd leading coefficient of $1$.  
Observe that there is a $g_{2}(x) \in \Zx$ with $g_{2}(x) \equiv \bar{g}_{2}(x) \pmod{2}$ and with each coefficient of
$f_{2}(x) - g_{2}(x)$ in $\{ 0,1 \}$.  In particular, $g_{2}(x)$ has degree $n$, and we see that 
\[
L(f-g_{2}) \le 1 + L(f_{2}-g_{2}) = 1 + L(\bar{f}_{2}-\bar{g}_{2}) \le 1 + (\ln n)^{2\ln(2)+\epsilon/2} \le (\ln n)^{2\ln(2)+\epsilon},
\] 
completing the proof.
\end{proof}

\section{Preliminaries to Theorem~\ref{thm:main1}}
Unless stated otherwise, we restrict our attention to arithmetic over $\F_{2}$, the field with two elements.  
In addition to the notation discussed in the previous section, we define the degree of a $0$ polynomial to be $-\infty$ with the understanding that
$\deg 0 = -\infty < \deg w$ for non-zero $w(x) \in \F_{2}[x]$.  

Our approach to proving Theorems \ref{thm:main1} relies on the following idea from \cite{DS18}. 
If $f(x)=\sum_{i=0}^{n}{a_ix^i}\in\Fx$ has degree $n$, then we define
\[
f_e(x)=\sum_{i=0}^{\lfloor n\slash 2\rfloor}{a_{2i}x^i}
\quad \text{ and } \quad
f_o(x)=\sum_{i=0}^{\lfloor(n-1)\slash 2 \rfloor}{a_{2i+1}x^i}.
\]
Observe that $f(x)=(f_e(x))^2+x(f_o(x))^2$. Further observe that $f'(x)=(f_o(x))^2$. 
As noted in \cite[Lemma~5.1]{DS18}, we have the following lemma.

\begin{lemma}\label{lemma:binary}
Let $f(x)\in\Fx$ with degree at least $2$.  The polynomial $f(x)$ is squarefree in $\Fx$ 
if and only if $\gcd(f_e(x),f_0(x))=1$.
 Moreover, any irreducible polynomial appearing as a factor of $f(x)$ to a multiplicity $> 1$ is a factor of the polynomial $\gcd(f_e(x),f_o(x))$.
\end{lemma}

This lemma will be crucial to our result. Observe that Lemma \ref{lemma:binary} allows one to view a polynomial $f(x)\in\Fx$ of degree $n$ as an ordered pair of polynomials of degree at most $n\slash 2$. Finding a nearby squarefree polynomial of degree $n$ is tantamount to finding a nearby ordered pair of polynomials which have trivial gcd.

We also make use of the following result.

\begin{lemma}\label{lemma:irr_prod}
Let $n \in \mathbb Z^{+}$, and let $p$ be a prime.   
The degree of the product of the monic irreducible polynomials of degree $\le n$ in $\Fpx$ is less than or equal to ${p(p^{n}-1)}/{(p-1)}$.
\end{lemma}

\begin{proof}
Every irreducible polynomial in $\Fpx$ of degree $n$ divides $x^{p^{n}} - x$.  
Hence, the degree of the product of the monic irreducible polynomials of degree $n$ is less than or equal to $p^n$.  
Since $p+p^{2}+ \cdots + p^{n} = p(p^{n}-1)/(p-1)$, the result follows.
\end{proof}

Next, we bound the minimum distance between a polynomial $f$ and a multiple of a polynomial $d$.

\begin{lemma}\label{lemma:congr}
Let $f(x),d(x)\in\Fx$ with $\deg d > 0$. 
There exists a polynomial $g(x)\in\Fx$ of degree at most $\deg f$ such that $d(x)|g(x)$ and $L_2(f-g) \le \deg d$.  
Furthermore, if also $\deg d \le \deg f$, then one can take $\deg g = \deg f$.  
\end{lemma}

\begin{proof}
There exist polynomials $q(x), r(x)\in\Fx$ such that $f(x)=d(x)q(x)+r(x)$,  $\deg r < \deg d$, and $\deg(d(x)q(x)) \le \deg f$, 
with equality if $\deg d \le \deg f$.
Since 
\[
L_2(f(x)-d(x) q(x))\le \deg d,
\] 
we can take $g(x)=d(x) q(x)$ to complete the proof.
\end{proof}

By taking $g(x) = d(x)q(x) + 1$ in the argument above, we obtain the following.

\begin{lemma}\label{lemma:not_congr}
Let $f(x),d(x)\in\Fx$ with $f(x)$ non-zero and $\deg d > 0$. 
There exists a polynomial $g(x)\in\Fx$ of degree at most $\deg f$ such that $\gcd(d,g)=1$ and $L_2(f-g) \le \deg d$.
Furthermore, if also $\deg d \le \deg f$, then one can take $\deg g = \deg f$.  
\end{lemma}

Here is another lemma that will prove useful later.

\begin{lemma}\label{lemma:cyclotomic}
For $t$ a positive integer, set $\Pi_{1} = \prod_{i=1}^{t}{(x^i+1)}\in\Fx$,
and let $\tilde{\Pi}_{1}$ be the product of the distinct irreducible polynomials dividing $\Pi_{1}$. 
The degree of $\tilde{\Pi}_{1}$ is $\le \lceil t/2 \rceil^2 - \lceil t/2 \rceil + 1$.
\end{lemma}

\begin{proof}
Each factor $x^{i}+1$ in $\Pi_{1}$ is divisible by $x+1$.  
Furthermore, if $i$ is even, then $x^i+1=(x^{i/2}+1)^2$ and thus does not contribute new irreducible factors to $\tilde{\Pi}_{1}$. 
In other words,
\begin{align*}
\deg \big( \tilde{\Pi}_{1} \big)
&\le 1 +  \deg\left( \prod_{i=1}^{\lceil t/2 \rceil} \dfrac{x^{2i-1}+1}{x+1} \right) \\[5pt]
&= 1 + 2 + 4  + 6 + \cdots + (2 \lceil t/2 \rceil - 2) \\[5pt]
&= 1 + 2 \big(1 + 2 + 3 + \cdots +  (\lceil t/2 \rceil -  1)\big)  \\[5pt]
&= 1 +\bigg(   \Big\lceil\frac{t}{2}\Big\rceil - 1  \bigg)   \Big\lceil\frac{t}{2}\Big\rceil,
\end{align*}
from which the lemma follows.
\end{proof}

We immediately have the following corollary.

\begin{corollary}\label{cor:cyclotomic}
Let $t$ be an integer $\ge 2$.  
Set $\Pi_{2}=x\prod_{i=1}^{t}{(x^i+x^{i-1}+\dots+x+1)}\in\Fx$, 
and let $\tilde{\Pi}_{2}$ be the product of the distinct irreducible polynomials dividing $\Pi_{2}$. 
The degree of $\tilde{\Pi}_{2}$ is $\le \lceil (t+1)/2 \rceil^2$.
\end{corollary}

\section{A proof of Theorem~\ref{thm:main1}}
To prove Theorem \ref{thm:main1}, we begin with a few technical lemmas.

\begin{lemma}\label{lemma:main1:1}
Fix $\epsilon \in (0,1)$, and let $n$ be a positive integer $\ge n_{0}(\epsilon)$ where $n_{0}(\epsilon)$ is sufficiently large.  
Set $t=\lceil 2\ln(\log_2 n)/(1-\epsilon)\rceil\in\N$. 
Let $\Pi_{2}$ be as in Corollary~\ref{cor:cyclotomic}.  
Let $f(x) \in\Fx$ with $\gcd(f(x),\Pi_{2})=1$ and $\deg f \le n$. 
Set 
\[
P(x) = P_{\epsilon}(x) = \prod_{\substack{p(x)\in\Fx \text{ irreducible}\\[2pt] \deg p \le t \\[2pt] p(x)\nmid f(x)}}{p(x)}.
\]
Then the polynomials in the collection
\[
\{f(x)+a(x)P(x)\}, \quad \text{where \ $a(x)\in\{1, x+1,x^2+x+1,\dots,x^t+x^{t-1}+\dots+x+1 \}$},
\] 
have no irreducible factors of degree $\le t$. Furthermore, the polynomials in this collection are pairwise coprime.
\end{lemma}

\begin{proof}
Let $p(x)$ be an irreducible polynomial of degree $\le t$.  
Then $p(x)|f(x)$ or $p(x)|P(x)$, but not both.
If $p(x)|P(x)$, then $p(x)\nmid f(x)$ so that $p(x)\nmid \big(  f(x)+a(x)P(x)  \big)$.
If $p(x)|f(x)$ then $p(x)\nmid P(x)$. 
In this case, since $p(x)|f(x)$ and $\gcd(f(x),\Pi_{2})=1$, we deduce that $\gcd(p(x), a(x)) = 1$. 
Therefore, $p(x)\nmid \big(  f(x)+a(x)P(x)  \big)$.
Thus, the polynomials of the form $f(x)+a(x)P(x)$, as defined above, have no irreducible factors of degree $\le t$. 

We deduce then that the polynomials of the form $f(x)+a(x)P(x)$ are pairwise relatively prime since they have no irreducible factors of degree less than or equal to $t$ 
and the difference of any two distinct $f(x)+a(x)P(x)$ is divisible only by irreducible polynomials of degree less than or equal to $t$.
\end{proof}

\begin{lemma}\label{lemma:main1:2}
Fix $\epsilon \in (0,1)$, and let $n$ be a positive integer $\ge n_{0}(\epsilon)$ where $n_{0}(\epsilon)$ is sufficiently large.  
Let $t=\lceil 2\ln(\log_2 n)/(1-\epsilon)\rceil\in\N$. 
Suppose that $f_0(x), f_1(x), \dots, f_t(x) \in \Fx$ are polynomials of degree $\le n$ 
which are also pairwise relatively prime 
and have no irreducible factors of degree $\le t$. 
If $g(x)\in\Fx$ has degree $\le n$, 
then there exists a polynomial $g_1(x)\in\Fx$ with $\deg g_{1} \le n$ such that $L_2(g-g_1) \le \log_2 n$ and, 
for some $i \in \{ 0, 1, 2, \ldots, t \}$, we have $\gcd(g_1,f_i)=1$.
Furthermore, if $\deg g \ge \log_2 n$, then we may take $\deg g_{1} = \deg g$.  
\end{lemma}

\begin{proof}
We proceed by adjusting the coefficients of $g(x)$ in the terms of degree $< \log_2 n$ to produce the desired $g_1(x)$. 
Observe that there are at least $2^{\log_2 n}=n$ such possibilities for $g_1(x)$. 
Furthermore, if $\deg g \ge \log_{2}n$, then each such $g_{1}(x)$ satisfies $\deg g_{1} = \deg g$.  
We examine the possible irreducible polynomials $w(x)$ which can divide $\gcd(g_1,f_i)$. 
By the assumptions on the $f_i(x)$, we see that $\deg w>t$. 

We consider now two cases depending on whether (i) $t<\deg w\le\log_2 n$ or (ii) $\deg w>\log_2 n$.
After considering both cases, we combine information from the two cases to obtain the desired result.

Case (i):  
Let $d = \deg w$. For each fixed choice of the coefficients, say $a_{j} \in \{ 0, 1 \}$, of $x^{j}$ in $g_{1}(x)$ for $j \in \{ d, d+1, \ldots, \lfloor \log_{2} n \rfloor \}$, 
there is at most one choice
of the coefficients $a_{j} \in \{ 0, 1 \}$ of $x^{j}$ in $g_{1}(x)$ for $j \in \{ 0, 1, \ldots, d-1 \}$ such that $g_{1}(x)$ is divisible by $w(x)$.  
Thus, such a $w(x)$ divides at most $2^{(\log_2 n)-d+1}$ possibilities for $g_{1}(x)$. 

Since every irreducible polynomial in $\Fx$ of degree $d$ divides $x^{2^{d}} - x$, 
there are at most ${2^d}/{d}$ irreducible polynomials of degree $d$ in $\Fx$. 
Therefore, there are at most 
\[
\dfrac{2^{d}}{d} \times 2^{(\log_2 n)-d+1} = \dfrac{2n}{d}
\]
possibilities for $g_{1}(x)$ that are divisible by an irreducible polynomial of degree $d$.  
By summing over $d$ in the range $(t, \log_2 n]$, we deduce that there are at most 
\[
\sum_{t<d\le\log_2 n}{\frac{2n}{d}} \le 2n\left(\ln(\log_2 n)-\ln(t)+O\left(\frac{1}{t}\right)\right) 
\le 2n\ln(\log_2 n)
\]
possibilities for $g_{1}(x)$ having an irreducible factor $w(x)$ as in (i).  
As this estimate is $> n$, we need to revise this estimate.  
We explain next how to reduce the above estimate by a factor of $t+1$. 

Recall that we are wanting $\gcd(g_1,f_i)=1$ for some $i \in \{ 0, 1, 2, \ldots, t \}$ rather than for every such $i$.  
We choose the $i \in \{ 0, 1, 2, \ldots, t \}$ that minimizes the number of possibilities for $g_{1}(x)$ which are divisible by an irreducible 
$w(x) \in \Fx$ with $\deg w \in \big(t, \log_2 n\big]$ and $w(x) | f_{i}(x)$.
Since the $f_{j}(x)$ are pairwise relatively prime, 
we deduce that the number of possibilities for $g_{1}(x)$ with $\gcd(g_1,f_i)$ divisible by an irreducible 
$w(x) \in \Fx$ of degree $d \in \big(t, \log_2 n\big]$ is at most
\[
\dfrac{1}{t+1} \sum_{t<d\le\log_2 n}{\dfrac{2n}{d}}
\le \dfrac{2n\ln(\log_2 n)}{t+1}
\le (1-\epsilon)n.
\]
We proceed now to Case (ii) with this choice of $i$. 

Case (ii):  
In this case, we use that an irreducible polynomial with degree  $> \log_2 n$ can divide at most one possibility for $g_{1}(x)$. 
With $i$ as in Case (i), 
we see that $f_i(x)$ can have at most $n/{\log_2 n}$ distinct irreducible factors of degree greater than $\log_2 n$. 
Therefore, at most $n/{\log_2 n}$ possibilities for $g_{1}(x)$ have an irreducible factor of degree greater than $\log_2 n$ 
in common with $f_{i}(x)$.

By combining our estimates from Case (i) and (ii), 
we deduce that there is some $f_i(x)$ such that there are at most 
\[
(1-\epsilon)n + \dfrac{n}{\log_2 n}
\]
possibilities for $g_{1}(x)$ that share a non-constant factor with $f_i$. 
Therefore, with $n \ge n_{0}(\epsilon)$, there exists a $g_1(x) \in \Fx$ with $\deg g_1 \le n$ such that $L_2(g-g_1) \le \log_2 n$ and $\gcd(g_1,f_i)=1$ for some $i \in \{ 0, 1, \ldots, t \}$.
\end{proof}

Now we proceed with the proof of Theorem \ref{thm:main1}.

\begin{proof}[Proof of Theorem \ref{thm:main1}]
We take $n$ sufficiently large as stated in the theorem, and set $\epsilon' = \epsilon/(\epsilon+4\ln 2)$.  
Let $ t=\lceil 2\ln(\log_2 n)/(1-\epsilon')\rceil\in\N$, 
and let $\Pi_{2}$ be as in Corollary~\ref{cor:cyclotomic}.  
From Corollary~\ref{cor:cyclotomic}, we see that $\deg(\tilde{\Pi}_{2}) \le \lceil (t+1)/2 \rceil^2$.  
We apply Lemma~\ref{lemma:not_congr} using the polynomials $f_{e}(x)$ and $\tilde{\Pi}_{2}$ to deduce that
there exists $\tilde{f}(x)$ with $\deg \tilde{f} \le \lfloor n/2 \rfloor$ and $\gcd(\tilde{f}(x),\Pi_{2})=1$ such that 
\[
L_2(f_e - \tilde{f}) \le \lceil (t+1)/2 \rceil^2.
\]
Furthermore, if $\deg f_{e} \ge \lceil (t+1)/2 \rceil^2$, we can take $\deg \tilde{f} = \deg f_{e}$ and do so. 
Define $P(x) = P_{\epsilon'}(x)$ as in Lemma~\ref{lemma:main1:1}. 
By this lemma, the polynomials in $\{\tilde{f}(x)+a(x)P(x)\}$, where $a(x)\in\{1, x+1,x^2+x+1,\dots,x^t+x^{t-1}+\dots+x+1 \}$, 
have no irreducible factors of degree $\le t$. 
Furthermore, the polynomials in this collection are pairwise coprime. 
For $i \in \{ 0, 1, \ldots, t \}$, set $\tilde{f}_i=\tilde{f}+(x^i+x^{i-1}+\dots+x+1)P(x)$.  
Since $\tilde{f}_i$ has no irreducible factor of degree $\le t$, we have in particular that $\tilde{f}_i(0) \ne 0$.  
From Lemma~\ref{lemma:irr_prod}, we see that
\[
L_2(\tilde{f}-\tilde{f}_i) \le \deg (\tilde{f} - \tilde{f}_{i}) \le t + 2 \big( 2^{t} - 1 \big) < t + 4 (\log_2 n)^{2\ln(2)+\epsilon/2}, 
\quad \text{for $0 \le i \le t$}.
\]
By Lemma \ref{lemma:main1:2}, there is a polynomial $\tilde{g}_1(x)$ with 
$\deg \tilde{g}_1 \le \lfloor (n-1)/2 \rfloor$ 
such that $L_2(\tilde{g}_1-f_o) \le \log_2 n$ and $\gcd(\tilde{g}_1,f_i)=1$ for some $i \in \{ 0, 1, \ldots, t \}$.
Furthermore, we take as we can $\deg \tilde{g}_1 = \deg f_{o}$ if $\deg f_{0} \ge \log_{2} n$.  
With $i$ so fixed, we set $g(x) = \tilde{f}_i(x)^2 + x \tilde{g}_1(x)^2$. 
Observe that $\deg g\le n$ and $g(x)$ is squarefree by Lemma~\ref{lemma:binary}.  
The condition $\deg f = n$ implies that $\deg f_{i} = \deg f_{e}$ or $\deg \tilde{g}_1 = \deg f_{o}$ 
with both holding if $\deg f_{e}$ and $\deg f_{o}$ are both 
$\ge \max\{ \log_{2} n, \lceil (t+1)/2 \rceil^2 \}$.   This implies $\deg g = \deg f = n$.  
The estimate
\begin{align*}
L_2(f-g)&= L_2(\tilde{f}_i^2-f_e^2)+L_2(\tilde{g}_1^2-f_o^2) \\[5pt]
&=L_2(\tilde{f}_i-f_e)+L_2(\tilde{g}_1-f_o) \\[5pt]
&\le L_2(f_{e} - \tilde{f}) + L_2(\tilde{f} - \tilde{f}_i) + L_2(\tilde{g}_1-f_o) \\[5pt]
&\le \Big\lceil\frac{t+1}{2}\Big\rceil^2+t+4(\log_2 n)^{2\ln(2)+\epsilon/2} +\log_2 n<(\ln n)^{2\ln(2)+\epsilon}
\end{align*}
completes the proof of the theorem.
\end{proof}

\bibliography{sqfreepoly_2} 
\bibliographystyle{plain}

\end{document}